\documentclass[11pt]{article}%
\usepackage{setspace}
\usepackage[T1]{fontenc}
\usepackage{latexsym}
\usepackage[reqno]{amsmath}
\usepackage{amsfonts}
\usepackage[round,colon]{natbib}
\usepackage{amssymb}
\usepackage{amstext}
\usepackage{amsxtra}
\usepackage{amsthm}
\usepackage{fullpage}
\usepackage[dvips]{graphicx}
\usepackage{color}
\usepackage{epsfig}
\usepackage{amsmath,amssymb,amsthm}
\usepackage[latin1]{inputenc}

\newtheorem{theorem}{Theorem}

\newtheorem{remark}[theorem]{Remark}

\begin{document}


\title{A simple proof for the multivariate Chebyshev inequality}

\author{Jorge Navarro\footnote{Tel/fax numbers: 34 868883508/34 868884182, email address: jorgenav@um.es }\\
Facultad de Matematicas, Universidad de Murcia, 30100 Murcia, Spain.}

\maketitle


\begin{abstract}
In this paper a simple proof of the Chebyshev's inequality for
random vectors obtained by  \cite{C} is obtained. This inequality
gives a lower bound for the percentage of the population of an
arbitrary random vector $\mathbf{X}$ with finite mean
$\mu=E(\mathbf{X})$ and a positive definite covariance matrix
$V=Cov(\mathbf{X})$ whose Mahalanobis distance with respect to $V$
to the mean $\mu$ is less than a fixed value. The proof is based on
the calculation of the principal components.

\end{abstract}

\section{Introduction}

The very well known Chebyshev's inequality for random variables
provides a lower bound for the percentage of the population in a
given distance with respect to the population mean when the variance
is known. It can be obtained from the Markov's inequality which can
be stated as follows. If $Z$ is a non-negative random variable with
finite mean $E(Z)$ and $\varepsilon>0$, then
$$\varepsilon\Pr(Z\geq
\varepsilon)=\varepsilon\int_{[\varepsilon,\infty)}dF_Z(x)\leq
\int_{[\varepsilon,\infty)} x dF_Z(x)\leq \int_{[0,\infty)} x
dF_Z(x)=E(Z)$$%
(where $F_Z(x)=\Pr(Z\leq x)$ is the distribution function of $Z$),
that is,
\begin{equation}\label{Markov}
\Pr(Z\geq \varepsilon)\leq \frac{E(Z)}{\varepsilon}.
\end{equation}
Chebyshev's inequality is then obtained as follows. If $X$ is a
random variable with finite mean $\mu=E(X)$ and variance
$\sigma^2=Var(X)>0$, then by taking $Z=(X-\mu)^2/\sigma^2$ in (\ref{Markov}),
we get
\begin{equation}\label{Cheby1}
\Pr\left(\frac{(X-\mu)^2}{\sigma^2} \geq \varepsilon \right)\leq
\frac{1}{\varepsilon}
\end{equation}
for all $\varepsilon>0$. It can also be written as
$$\Pr((X-\mu)^2 < \varepsilon \sigma^2)\geq 1-
\frac{1}{\varepsilon}$$ or as
$$\Pr(|X-\mu|< r)\leq
1-\frac{\sigma^2}{r^2}$$ for all $r>0$.

There are several extensions of these results to the multivariate
case (see e.g. \cite{C,MO} and the references therein). Recently,
\cite{C} proved the following Chebyshev's inequality
$$\Pr( (\mathbf{X}-\mu)'V^{-1}(\mathbf{X}-\mu) \geq \varepsilon )\leq \frac n \varepsilon$$
for all $\varepsilon>0$ and for all random vectors
$\mathbf{X}=(X_1,\dots,X_n)'$ ($w'$ denotes the transpose of $w$)
with finite mean vector $\mu=E(\mathbf{X})$ and positive definite
covariance matrix
$V=Cov(\mathbf{X})=E((\mathbf{X}-\mu)(\mathbf{X}-\mu)')$. Extensions
of Chen's result to Hilbert-space-valued and Banach-space-valued
random elements can be seen in \cite{PR} and \cite{ZH},
respectively.

In this paper a new (in my knowledge) proof for Chen's result is
given. The proof is based on the calculation of the principal
components. The main advantage of the new proof is that it is so
simple that it can be can be included in all the basic multivariate
analysis text books. Some comments are also included after the
proof. In these comments, the case in which $|V|=0$ is analyzed.
Also some consequences in regression analysis are given.

\section{Main result}

\begin{theorem}
Let $\mathbf{X}=(X_1,\dots,X_n)'$ be a random vector
with finite mean vector $\mu=E(\mathbf{X})$ and positive definite
covariance matrix
$V=Cov(\mathbf{X})$. Then
\begin{equation}\label{CMul}
\Pr( (\mathbf{X}-\mu)'V^{-1}(\mathbf{X}-\mu) \geq \varepsilon )\leq \frac n \varepsilon
\end{equation}
for all $\varepsilon>0$
\end{theorem}

\begin{proof}
Let us consider the random variable
$$Z=(\mathbf{X}-\mu)'V^{-1}(\mathbf{X}-\mu).$$
As $V$ is positive definite, then $Z\geq 0$. Moreover, as $V$ is also symmetric, there exists an ortogonal matrix $T$ such that $TT'=T'T=I_n$ and $T'VT=D$, where $I_n$ is the identity matrix of order $n$ and $D=diag(\lambda_1,\dots,\lambda_n)$ is the diagonal matrix with the ordered eigenvalues  $\lambda_1\geq \dots \geq \lambda_n>0$. Then $V=TDT'$ and $V^{-1}=TD^{-1}T'$. Therefore
$$Z=(\mathbf{X}-\mu)'TD^{-1}T'(\mathbf{X}-\mu)=[D^{-1/2}T'(\mathbf{X}-\mu)]'
[D^{-1/2}T'(\mathbf{X}-\mu)]=\mathbf{Y}'\mathbf{Y},$$
where
$$\mathbf{Y}=D^{-1/2}T'(\mathbf{X}-\mu)$$
and  $D^{-1/2}=diag(\lambda_1^{-1/2},\dots,\lambda_n^{-1/2})$. The random vector $\mathbf{Y}$ satisfies
$$E(\mathbf{Y})=E(D^{-1/2}T'(\mathbf{X}-\mu))=D^{-1/2}T'E(\mathbf{X}-\mu)=0$$
and
$$Cov(\mathbf{Y})=Cov(D^{-1/2}T'(\mathbf{X}-\mu))=D^{-1/2}T'VTD^{-1/2}
=D^{-1/2}DD^{-1/2}=I_n.$$
Therefore
$$E(Z)=E(\mathbf{Y}'\mathbf{Y})=E\left(\sum_{i=1}^n Y_i^2\right)=\sum_{i=1}^n E(Y_i^2)=\sum_{i=1}^n Var(Y_i)=n.$$
Hence, from Markov's inequality (\ref{Markov}), we get
$$\Pr(Z\geq \varepsilon)=\Pr((\mathbf{X}-\mu)'V^{-1}(\mathbf{X}-\mu)\geq \varepsilon)\leq \frac{E(Z)}{\varepsilon}=\frac n {\varepsilon}$$
for all $\varepsilon>0$.
\end{proof}

\begin{remark}
Of course, if $n=1$ in (\ref{CMul}), then the univariate Chebyshev inequality (\ref{Cheby1}) is obtained. The vector $\mathbf{Y}=D^{-1/2}T'(\mathbf{X}-\mu)$ used in the preceding proof is the vector of the standardized principal components of $\mathbf{X}$. The inequality in (\ref{CMul}) can also be written as
\begin{equation} 
\Pr( (\mathbf{X}-\mu)'V^{-1}(\mathbf{X}-\mu)< \varepsilon )\geq 1- \frac n \varepsilon
\end{equation}
for all $\varepsilon>0$. This inequality says that the ellipsoid
$$E_\varepsilon=\{\mathbf{x}\in \mathbb{R}^n: (\mathbf{x}-\mu)'V^{-1}(\mathbf{x}-\mu)< \varepsilon\}$$
contains at least the $100(1- n/\varepsilon)\%$ of the population for all $\varepsilon\geq n$
for any random vector $\mathbf{X}$. It is well known that the principal components coincide with the projections to the principal axes of that ellipsoid. For example, for $\varepsilon=4n$, we have
$$\Pr( (\mathbf{X}-\mu)'V^{-1}(\mathbf{X}-\mu)< 4n )\geq 0.75.$$
The inequality can also be written as
$$\Pr(d_V(\mathbf{X},\mu)<r) \geq 1- \frac n {r^2},$$
where
$$d_V(x,y)=\sqrt{(\mathbf{x}-\mathbf{y})'V^{-1}(\mathbf{x}-\mathbf{y})}$$
is the Mahalanobis distance associated with the positive definite
matrix $V$. Hence (\ref{CMul}) gives a lower bound for the
percentage of points in spheres ``around'' the mean in the
Mahalanobis distance. A comparison between the volume in these
spheres and that in the regions containing the same probability in
other multivariate Chebyshev inequalities can be seen in \cite{C}.
\end{remark}

\begin{remark}
Recall that in the preceding theorem $\mathbf{X}$ is an arbitrary non-singular random vector with finite mean and finite variances-covariances. In particular, if $\mathbf{X}$ has a normal distribution, then the ellipsoid $E_\varepsilon$ coincides with regions determined by the level curves of the normal probability density function. Moreover, in this case, the exact probability can be obtained by using that $\mathbf{Y}$ is normally distributed with mean $E(\mathbf{Y})=0$ and $Cov(\mathbf{Y})=I_n$. Hence $Y_1,\dots,Y_n$ are independent and identically distributed with a common standard normal distribution and
$$Z=\mathbf{Y}'\mathbf{Y}=Y_1^2+\dots+Y_n^2$$
has a chi-squared distribution with $n$ degrees of freedom (a well known result, see, e.g., page 39 in \cite{MKB}). For example, for $n=2$, we obtain
$$\Pr( (\mathbf{X}-\mu)'V^{-1}(\mathbf{X}-\mu)< 8 )=\Pr(\chi_2^2<8)=0.9816844\geq 0.75.$$
\end{remark}

\begin{remark}
If $|V|=0$ and $\mathbf{X}$ is non-degenerate, that is, $$\lambda_1\geq \dots\geq \lambda_{m-1}>\lambda_m=\dots=\lambda_n=0$$ for an $m\in\{2,\dots,n\}$, then we can consider $\mathbf{Y}=BT'(\mathbf{X}-\mu)$, where $B=diag(\lambda_1^{-1/2},\dots,\lambda_{m-1}^{-1/2},0,\dots,0)$ and by using the preceding theorem we obtain
$$\Pr( (\mathbf{X}-\mu)'TCT'(\mathbf{X}-\mu)< \varepsilon )\geq 1- \frac {m-1} \varepsilon,$$
where $C=diag(\lambda_1^{-1},\dots,\lambda_{m-1}^{-1},0,\dots,0)$. This inequality says that the ellipsoid  on the region determined by the point $\mu$ and the $m-1$ first principal components contains at least the $100(1-(m-1)/\varepsilon)\%$ of the population opf the random vector $\mathbf{X}$.
\end{remark}

\begin{remark}
The inequality in (\ref{CMul}) can be applied to conditional random vectors. For example, if we consider $(\mathbf{X},\mathbf{Y})$ where $\mathbf{X}=(X_1,\dots,X_k)$ and $\mathbf{Y}=(X_{k+1},\dots,X_n)$, $\mu(\mathbf{x})=E(\mathbf{Y}|\mathbf{X}=\mathbf{x})$ is finite and
$V(\mathbf{x})=Cov(\mathbf{Y}|\mathbf{X}=\mathbf{x})$ is a positive definite matrix for a fixed $\mathbf{x}$, then
\begin{equation} 
\Pr( [\mathbf{Y}-\mu(\mathbf{x})]'[V(\mathbf{x})]^{-1}[\mathbf{Y}-\mu(\mathbf{x})] \geq \varepsilon |\mathbf{X}=\mathbf{x})\leq \frac {n-k} \varepsilon
\end{equation}
for all $\varepsilon>0$. This inequality gives a confidence region around the regression map
$\mu(\mathbf{x})=E(\mathbf{Y}|\mathbf{X}=\mathbf{x})$. Similar results can be obtained for other conditional random vectors as, e.g., $(\mathbf{Y}|\mathbf{X}\geq \mathbf{x})$.
\end{remark}


\section*{Acknowledgements}

This work is partially supported by Ministerio de Ciencia y
Tecnología de España under grant MTM2009-08311-FEDER, Ministerio de
Economía y Competitividad under grant MTM2012-34023-FEDER  and
Fundación Séneca of C.A.R.M. under grant 08627/PI/08.







\end{document}